\documentclass[12pt,reqno]{amsart}
\usepackage{geometry}                
\usepackage{subfig}
\usepackage{amsmath,amssymb}
\usepackage{graphicx}
\usepackage{amssymb,latexsym, amsmath}
\usepackage{amsfonts,amsthm}
\usepackage{amscd}
\usepackage{mathrsfs}
\usepackage{hyperref}
\usepackage{eucal}
\usepackage{multicol}
\usepackage{epstopdf}
\usepackage{amsgen}
\usepackage{xspace}
\usepackage{verbatim}
\usepackage{stmaryrd}
\usepackage{graphicx}
\usepackage{caption}

\newcommand{\gw}{\Omega}

\newcommand{\ga}{\gamma}

\newcommand{\gb}{\beta}
\newcommand{\G}{\Gamma}
\newcommand{\gl}{\lambda}

\newcommand{\gs}{\sigma}

\newcommand{\mb}{\mathbb}
\newcommand{\ms}{\mathscr}

\newcommand{\vp}{\varphi}
\newcommand{\ve}{\varepsilon}

\newcommand{\beq}{\begin{equation}}
\newcommand{\eeq}{\end{equation}}
\newcommand{\bea}{\begin{align}}
\newcommand{\eea}{\end{align}}
\newcommand{\bthm}{\begin{theorem}}
\newcommand{\ethm}{\end{theorem}}
\newcommand{\bpr}{\begin{proof}}
\newcommand{\epr}{\end{proof}}
\newcommand{\bcl}{\begin{corollary}}
\newcommand{\ecl}{\end{corollary}}
\newcommand{\bpn}{\begin{proposition}}
\newcommand{\epn}{\end{proposition}}
\newcommand{\bre}{\begin{remark}}
\newcommand{\ere}{\end{remark}}
\newcommand{\bdf}{\begin{definition}}
\newcommand{\edf}{\end{definition}}
\newcommand{\bss}{\begin{align*}}
\newcommand{\ess}{\end{align*}}

\newcommand{\bl}{\label}

\newtheorem{theorem}{Theorem}[section]
\newtheorem{corollary}[theorem]{Corollary}

\newtheorem{proposition}[theorem]{Proposition}

\theoremstyle{definition}
\newtheorem{definition}[theorem]{Definition}
\theoremstyle{remark}
\newtheorem{remark}{Remark}

\numberwithin{equation}{section}

\begin{document}

\title[Synchronization of Memristive Hopfield Neural Networks]{Approximate Synchronization of Memristive Hopfield Neural Networks}

\author[Y. You]{Yuncheng You}
\address{University of South Florida, Tampa, FL 33620, USA}
\email{yygwmp@gmail.com}

\thanks{}

\subjclass[2020]{34D06, 34D45, 37N25, 68T07, 92B20}

\date{\today}


\keywords{Hopfield neural network, approximate synchronization, memristive coupling, Hebbian weight dynamics, network exponential convergence.}

\begin{abstract} 
Asymptotic synchronization is one of the essential differences between artificial neural networks and biologically inspired neural networks due to mismatches from dynamical update of weight parameters and heterogeneous activations. In this paper a new concept of approximate synchronization is proposed and investigated for Hopfield neural networks coupled with nonlinear memristors.  It is proved that global solution dynamics are robustly dissipative and a sharp ultimate bound is acquired. Through \emph{a priori} uniform estimates on the interneuron differencing equations, it is rigorously shown that approximate synchronization to any prescribed small gap at an exponential convergence rate of the memristive Hopfield neural networks occurs if an explicitly computable threshold condition is satisfied by the interneuron coupling strength coefficient. The main result is further extended to memristive Hopfield neural networks with Hebbian learning rules for a broad range of applications in unsupervised train learning.
\end{abstract}

\maketitle
 
\section{Introduction}

Artificial neural networks are used to somehow emulate the dynamic behavior and functions of biological neurons and nervous systems through a variety of mathematical models and computational algorithms \cite{HJ, Lin, JS}. Hopfield neural networks (HNN) initially proposed in 1982 \cite{Hp} is a recursive or recurrent artificial neural networks \cite{OC}. HNN models coupled with memristor-based synapses have been actively studied and achieved rich and important results in the recent decade \cite{MI, D, Yu}.

Memristor as a physics concept, which relates the electric charge of circuit and the electromagnetic flux expressed by time integral of the voltage, was initiated by Chua \cite{Chua}. It is recognized \cite{V, Xu, VM} that memristor is a promising and ideal element in spiking neural networks to mimic the learning and memorizing functions of human brain. Many researches are reported in recent years on memristive biological neural network synchronization \cite{Lin, Y1, Y2, YTT}, image encryption \cite{D, W}, nanotechnology \cite{E}, and quantum computers \cite{Li}.

Synchronization dynamics is an important topic on cognitive functions and information processing for biological neural networks, chaotic systems, and deep learning \cite{Ay, BK, WD}. For biological neural network models in terms of Hindmarsh-Rose equations \cite{Y2}, FitzHugh-Nagumo equations \cite{YTT}, and reaction-diffusion equations \cite{YT} with memristive synapses, or cellular models with boundary feedback \cite{LY}, complete synchronization (meaning interneuron gaps for all solutions converge to zero) can be triggered if the network coupling strength satisfies a threshold condition, because in these models all the nodes of the network are represented by a system of qualitatively identical differential or difference equations. 

In the contrast, for general artificial neural networks used in machine learning and AI applications, such complete synchronization cannot be achieved because of the mismatched weight dynamics as well as the heterogeneous activation functions which are inherently characterized. The mismatching component equations of different nodes in artificial neural network model will inevitably cause nonzero asymptotic errors and possibly incomplete synchronization of the network. 

It is essentially open to establish an analytic theory on approximate synchronization for the typical ODE models \cite{Ag} of artificial neural networks, in view of currently sporadic researches reported on particular reaction-diffusion neural networks of FHN neurons, cf. \cite{JT} and references therein.

In this paper we shall quantitatively define approximate synchronization for artificial neural network models and investigate a novel model of memristive Hopfield neural networks (mHNN) featuring nonlinear memristor-potential coupling and weak synapses among neuron nodes. We shall rigorously prove a main result that approximate synchronization of the mHNN occurs if the coefficient of network coupling strength satisfies a computable threshold condition. The extension to another new model of memristive HNN with Hebbian learning rules is also achieved.

Consider a memristive Hopfield neural network composed of $m$ neuron nodes, denoted by $\mathcal{NW} = \{\mathcal{N}_i : i = 1, 2, \cdots, m\}$, where $m \geq 2$ is a positive integer. Each neuron $\mathcal{N}_i, 1 \leq i \leq m$, and a memristor characterized by its memductance $\rho$ are presented by the following differential equations: 
\beq \bl{Meq}
	\begin{split}
	\frac{d u_i}{d t} & = - a_i u_i + \sum_{j =1}^m w_{ij} \, f_j (u_j) + k\, \vp_i (\rho) u_i + J_i - Pu_i  \sum_{j=1}^m \Gamma (u_j), \;\; 1 \leq i \leq m,  \\
	\frac{d \rho}{d t} & = \sum_{i=1}^m \ga_i u_i - b \rho,  \quad  t > 0.
	\end{split} 
\eeq
The coefficient $P > 0$ is called the network coupling strength and the interneuron synaptic function
\beq \bl{Ga}
	\G (s) = \frac{1}{1 + \exp [- r(s - V)]}, \quad s \in \mb{R},
\eeq
is a sigmoidal function, in which the parameter $r > 0$ shapes the slope of $\G (s)$ and the parameter $V \in \mathbb{R}$ is a switch for neuron bursting. The state variable $u_i(t)$ simulates the electrical potential of the $i$th neuron $\mathcal{N}_i$, for $1 \leq i \leq m$. The initial states of the ODE system \eqref{Meq} will be denoted by 
\beq \bl{inc}
	 u_i^0 = u_i(0),  \quad \rho^0 = \rho (0), \quad 1 \leq i \leq m.
\eeq
We denote the initial values in \eqref{inc} by $g^0 = \text{col}\; (u_1^0, \, u_2^0, \cdots, u_m^0, \, \rho^0) \in \mathbb{R}^{m+1}$. The memristive feature of this model is reflected by the memductance-potential coupling term $k \vp_i (\rho) u_i$ and the weak synapses among neurons is shown by the nonlinear sigmoidal coupling $Pu_i \, \sum_{j=1}^m \Gamma (u_j)$. 

For this mHNN model, we notice that  the interaction weights $\{w_{ij}: 1 \leq i, j \leq m\}$ and the activation functions $\{f_j(u_j): 1 \leq j \leq m\}$ can well be different \cite[Chapter 2]{OC}, while the memristor window functions $\vp_i (\rho)$ for different neuron nodes $\mathcal{N}_i$ can also be heterogeneous \cite[Chapter 2]{VM}. It implies that complete synchronization for such artificial neural networks in general will not be possible and approximate synchronization is the aim to research.

As a setting of the problem, we make the following Assumption: The scalar functions $f_i (s)$ and $\vp_i (s), 1 \leq i \leq m$ are locally Lipschitz continuous functions and satisfy the following conditions: 
\beq \bl{Asp}
	a_i > k, \quad |f_i(s)|  \leq \gb \;\; \text{and} \;\; \vp_i (s) = 1 - \eta_i s^2,  \;\; \text{for} \; s \in \mathbb{R}, \; 1 \leq i \leq m.
\eeq
Assume that  the parameters $a_i, b, \gb, \eta_i$ and the memristive coupling strength $k$ can be any given positive constants, while weight coefficients $w_{ik}$, input  potential constants $J_i$ and coefficients $\ga_i$ can be any given real numbers. 

Note that the window function $k \vp_i (\rho)$ set in \eqref{Asp} is adaptable to deal with other types of memristors such as linear, hyperbolic-tangent, or Titanium dioxide memristor with the window function $\Hat{\vp} (s) = c s(1 - s)$ \cite{VM}. 

\begin{definition}
	An artificial neural network with a mathematical model of differential equations in a state space $Z$ such as the mHNN $\mathcal{NW}$ is said to be approximately synchronizable if there exists a sufficient threshold condition on the network coupling parameter $\mathcal{P}$ such that for any prescribed gap $\ve > 0$ there is a threshold value $\mathcal{P}(\ve)$ to trigger an approximate convergence with the network synchronous degree 
\beq \bl{gap}
	\deg_s \,(\mathcal{NW})= \sup_{g^0 \in Z} \left\{ \max_{1 \leq i < j \leq m} \left\{\limsup_{t \to \infty} \, |u_i (t) - u_j(t)| \right\}\right\} < \ve,
\eeq
where $u_i (t)$ is the state variable of the $i$th neuron node, $1 \leq i  \leq m$. If the limsup convergence in \eqref{gap} admits a uniform exponential rate, then it is called approximate exponential synchronization.
\end{definition}

The rest of the paper is organized as follows. In Section 2 we analyze dissipative dynamics of the solutions of the initial value problem \eqref{Meq}-\eqref{inc} and prove the existence of absorbing set in the state space. In Section 3 we shall prove the main result on approximate exponential synchronization of this memristive Hopfield neural network based on sharp uniform estimates of the interneuron differencing equations. In Section 4 the result is generalized to memristive Hopfield-Hebbian neural networks link to a broad range of unsupervised learning applications. 

\section{\textbf{Dissipative Dynamics and Global Attractor}}

The mathematical model \eqref{Meq}-\eqref{inc} can be formulated into an initial value problem of the evolutionary equation:
\begin{equation} \label{pb}
\begin{split}
	\frac{d g}{d t} \,& = A g + F(g),   \; \;  t > 0,  \\
	& g(0) = g^0 \in \mathbb{R}^{m+1}.
\end{split}
\end{equation}
In \eqref{pb} the column vector function $g(t) = \text{col}\; (u_1 (t), u_2 (t), \cdots, u_m (t), \rho(t))$ and the initial state is
$$
	g(0) = g^0 = \text{col}\; (u_1^0, \, u_2^0, \cdots, u_m^0, \, \rho^0).
$$
The norm $\|g(t)\|$ in the state space $\mathbb{R}^{m+1}$ for \eqref{pb} is defined to be $\|g(t)\|^2 = \sum_{i=1}^m |u_i(t)|^2 + |\rho (t)|^2$. In \eqref{pb} the diagonal square matrix 
\begin{equation} \label{opA}
	A = diag \, (- a_1, - a_2, \cdots, -a_m, - b)
\end{equation}
and the nonlinear vector  
\begin{equation} \label{opf}
F(g) =
\begin{pmatrix}
	\sum_{j =1}^m w_{1j} f_j (u_j) + k\, \vp_1 (\rho) u_1 + J_1 - Pu_1 \, \sum_{j=1}^m \Gamma (u_j) \\[7pt]
	\sum_{j =1}^m w_{2j} f_j (u_j) + k\, \vp_2 (\rho) u_2 + J_2 - Pu_2 \, \sum_{j=1}^m \Gamma (u_j) \\[4pt]
	\vdots\\[4pt]
	\sum_{j =1}^m w_{mj} f_j (u_j) + k\,\vp_m (\rho) u_m + J_m - Pu_m \, \sum_{j=1}^m \Gamma (u_j) \\[7pt]
	\sum_{i=1}^m \ga_i u_i
\end{pmatrix}
\end{equation}
is a locally Lipschitz continuous vector function. 

As preliminaries cf.\cite{CV, GY}, a general dynamical system only for time $t \geq 0$ can be called semiflow. For a semiflow denoted by  $\{S(t)\}_{t \geq 0}$ on a Banach space $\ms{X}$, a bounded subset $B^* \subset \ms{X}$ is called absorbing set if for any given bounded set $B$ in the space $\ms{X}$ there is a finite time $T_B \geq 0$ such that $S(t)B \subset B^*$ for all $t  > T_B$. A semiflow is called dissipative if there exists an absorbing set in the state space $\ms{X}$, which also implies the existence of a global attractor if $\ms{X}$ is finite dimensional.

The Young's inequality are used widely in differential inequality management. It states that for any positive numbers $x$ and $y$, if $\frac{1}{p} + \frac{1}{q} = 1$ and $p > 1, q > 1$, then
\beq \bl{Yg}
	x\,y \leq \frac{1}{p} \ve x^p + \frac{1}{q} C(\ve, p)\, y^q \leq \ve x^p + C(\ve, p)\, y^q, 
\eeq
where $C(\ve, p) = \ve^{-q/p}$ and the constant $\ve > 0$ can be arbitrarily given. 

\begin{theorem} \label{T1}
	For any given initial state $g^0 \in \mathbb{R}^{m+1}$, there exists a unique global solution $g(t; g^0)$ in time $t \in [0, \infty)$ for the initial value problem \eqref{pb} of the memristive Hopfield neural network $\mathcal{NW}$ described by the model \eqref{Meq} and the Assumption \eqref{Asp}. 
\end{theorem}

\begin{proof}
Since the nonlinear vector function $F(g)$ in the semilinear equation \eqref{pb} is locally Lipschitz continuous, there exists a unique local solution $g(t; g^0)$ in time to the initial value problem \eqref{pb}. 

Multiply the $u_i$-equation in \eqref{Meq} by $C_1 u_i(t)$ for $t \in [0, T_{max})$, which is the maximal existence interval of that solution. Then sum them altogether for $1 \leq i \leq m$. By the Assumption \eqref{Asp} we get		
\begin{equation} \bl{ui}
	\begin{split}
	&\frac{1}{2} \frac{d}{dt} \sum_{i = 1}^m C_1 u_i^2 (t) \leq \frac{1}{2} \frac{d}{dt} \sum_{i = 1}^m C_1 u_i^2 (t) + C_1 \sum_{i=1}^m \sum_{j=1}^m \frac{Pu_i^2}{1 + \exp [- r(u_j - V)]}  \\
	= &\, \sum_{i=1}^m C_1 \left[- a_i u^2_i (t) + \sum_{j =1}^m w_{ij} f_j (u_j(t)) u_i(t) + k \vp_i (\rho (t)) u_i^2(t) +  J_i u_i(t) \right]  \\
	\leq &\, \sum_{i=1}^m C_1 \left[- a u_i^2(t) + m W\gb |u_i (t)| + k \left(1 - \eta_i \rho^2(t)\right) u_i^2(t) + J |u_i(t)| \right]   \\
	\leq &\, \sum_{i=,1}^m C_1\left[- (a - k) u_i^2(t) + (m W \gb + J ) |u_i (t)| \right] \\[4pt]
	\leq &\, \sum_{i=,1}^m C_1\left[- \frac{1}{2}(a - k) u_i^2(t) + \frac{(m W \gb + J )^2}{2(a - k)} \right], \quad \text{for} \;\; t \in [0, T_{max}),
	\end{split}
\end{equation} 
where $a = \min \{a_i: 1 \leq i \leq m\} > k$, the double summation term on the right-hand side of the first inequality is nonnegative, $- k\eta_i \rho^2(t) u_i^2(t) \leq 0$, and a square completion is used in the last step in \eqref{ui}. The scaling constant $C_1 > 0$ is to be specified later and the new positive parameters $a = \min\, \{ a_i: 1 \leq i \leq m\}, \, W = \max \left\{|w_{ij} |: 1 \leq i, j \leq m \right\}, \, J = \max \left\{ |J_i|: 1 \leq i \leq m \right\}$.

Then multiply the $\rho$-equation in \eqref{Meq} with $\rho (t)$, for $t \in [0, T_{max})$. By the Young's inequality \eqref{Yg}, we have

\beq \bl{rho}
	\begin{split}
	&\frac{1}{2} \frac{d}{dt} \rho^2 (t) = \sum_{i=1}^m \ga_i \,u_i(t) \rho(t) - b \rho^2(t)  \\[2pt]
	\leq &\, \sum_{i=1}^m \left[\frac{m \ga_i^2}{2b} u_i^2(t) + \frac{b}{2m} \rho^2(t) \right] - b \rho^2(t) = \sum_{i=1}^m \frac{m \ga_i^2}{2b} u_i^2 (t) - \frac{b}{2} \rho^2(t), \quad  t \in [0, T_{max}).
	\end{split}
\eeq
Add the above two inequalities \eqref{ui} and \eqref{rho}. We come up with 
\beq \bl{ur}
	\begin{split}
	& \frac{1}{2}\frac{d}{dt} \left(\sum_{i = 1}^m C_1 u_i^2(t) + \rho^2(t)\right)   \\
        \leq \sum_{i=1}^m C_1 &\left[ - \frac{1}{2}(a - k) u_i^2(t) + \frac{(m W \gb + J )^2}{2(a - k)}\right] +  \sum_{i=1}^m \frac{m \ga_i^2}{2b}\, u_i^2 (t) - \frac{b}{2}\, \rho^2(t)   \\
        = - \frac{1}{2} \sum_{i=1}^m &\left[ C_1 (a - k) - \frac{m \ga_i^2}{b}\right] u_i^2(t) - \frac{b}{2}\,\rho^2(t) + \frac{C_1 m (m W \gb + J )^2}{2(a - k)}, 
	\end{split}
\eeq
for $t \in [0, T_{max})$, which is the maximal existence interval of a local solution $g(t; g^0)$ of this initial value problem \eqref{pb} and may depend on the initial state $g^0$. 

Now we can choose the scaling constant to be
\beq \bl{C1}
	C_1 = \frac{1}{(a - k)} \left(\frac{m \ga_i^2}{b} + b \right) \quad \text{so that}  \quad  C_1(a - k) - \frac{m \ga_i^2}{b} = b.
\eeq
With this choice, from \eqref{ur} it follows that
\beq \bl{up}
	\begin{split}
	& \frac{d}{dt} \left(\sum_{i = 1}^m C_1 u_i^2(t) + \rho^2(t)\right) + b \left(\sum_{i=1}^m u_i^2(t) + \rho^2(t) \right)  \\ 
        \leq &\, \frac{1}{a - k} \left(\frac{m \ga_i^2}{b} + b \right) \frac{m (m W \gb + J )^2}{a - k},  \quad  \text{for} \;  t \in [0, T_{max}).
	\end{split}
\eeq
Denote the positive constant on the right-hand side of \eqref{up} by
\beq \bl{C2}
	C_2 = \left(\frac{m \ga_i^2}{b} + b \right) \frac{m (mW \gb + J)^2}{(a - k)^2}.
\eeq
Then \eqref{up} gives rise to a Gronwall-type differential inequality: 
\beq \bl{GZ}
	\frac{d}{dt} \left(\sum_{i = 1}^m C_1 u_i^2(t) + \rho^2(t)\right) + \mu \left(\sum_{i=1}^m C_1 u_i^2(t) + \rho^2(t) \right) \leq C_2,  \;\; \text{for} \; t \in [0, T_{max}),
\eeq
where 
\beq \bl{mu}
	\mu = b \,\min \left\{\frac{1}{C_1}, \; 1 \right\}.
\eeq
We can solve this Gronwall differential inequality \eqref{GZ} to obtain the following bounding estimate of all the solutions to the governing equation \eqref{pb} of this memristive Hopfield neural network model \eqref{Meq}, 
$$
	 \left(\sum_{i=1}^m C_1 u_i^2(t) + \rho^2(t) \right) \leq e^{- \mu t} \left(\sum_{i=1}^m C_1 u_i^2(0) + \rho^2(0) \right) + \frac{C_2}{\mu}, 
$$
so that
\beq \label{dse}
	\|g(t; g^0)\|^2 = \sum_{i=1}^m |u_i (t)|^2 + |\rho (t)|^2 \leq \frac{\max \{C_1, 1\}}{\min \{C_1, 1\}} e^{- \mu t} \|g^0\|^2 +  \frac{C_2}{\mu \min \{C_1, 1\}}, \;\, t \in [0, \infty).
\eeq
It is also shows that the existence interval  $[0, T_{max}) = [0, \infty)$ for every solution with any initial state $g^0 \in \mathbb{R}^{m+1}$, because the solution $g(t; g^0)$ will never blow up at any finite time. Therefore, under the Assumption \eqref{Asp}, for any given initial state there exists a unique global solution for $t \in [0, \infty)$ for the mHNN model \eqref{Meq}.
\end{proof}

Based on the global existence of solutions shown in Theorem \ref{T1}, the solution semiflow $\{S(t): g^0 \longmapsto g(t; g^0)\}_{t \geq 0}$ on the state space $\mathbb{R}^{m+1}$ can be called the mHNN semiflow. The next theorem shows that $\{S(t)\}_{t \geq 0}$ is a dissipative dynamical system and there exists a global attractor in the state space. 

\begin{theorem} \label{T2}
	There exists a bounded absorbing set for the mHNN semiflow $\{S(t)\}_{t \geq 0}$ in the state space $\mathbb{R}^{m+1}$, which is the bounded ball 
\beq \label{Br}
	B^* = \{ g \in \mathbb{R}^{m+1}: \| g \|^2 \leq Q\}.
\eeq 
Here the uniform constant 
\beq \bl{K}
	Q = 1 +  \frac{C_2}{\mu \min \{C_1, 1\}}
\eeq
is independent of any initial state. The constants $C_1$ and $C_2$ are given in \eqref{C1} and \eqref{C2}. Moreover, there exists a global attractor $\ms{A}$ for this mHNN semiflow.
\end{theorem}

\begin{proof}
This is the consequence of the global uniform estimate \eqref{dse} shown above, which implies that, since $e^{- \mu t} \to 0$ as $t \to \infty$,
\beq \label{lsp}
	\limsup_{t \to \infty} \|g(t; g^0)\|^2 = \limsup_{t \to \infty} \left[ \sum_{i=1}^m |u_i (t)|^2 + |\rho (t)|^2 \right] < Q 
\eeq
for all the solutions of \eqref{pb} with any initial state $g^0$. Moreover, \eqref{dse} shows that for any given bounded set $B = \{g \in \mathbb{R}^{m+1}: \|g \|^2 \leq L\}$ in the state space, there exists a finite time
\beq \bl{TB}
	T_B = \frac{1}{\mu} \log^+ \left(L\,\frac{\max \{C_1, 1\}}{\min \{C_1, 1\}}\right) \geq 0
\eeq
such that all the solution trajectories started at the initial time $t = 0$ from inside the set $B$ will permanently enter the bounded ball $B^*$ shown in \eqref{Br} for all $t > T_B$.  Therefore, the bounded ball $B^*$ is an absorbing set for the semiflow $\{S(t)\}_{t \geq 0}$ so that this mHNN semiflow is dissipative dynamical system..

Finally, since the state space $\mathbb{R}^{m+1}$ is a finite-dimensional Hilbert space, the closed bounded absorbing set $B^*$ is a compact set. According to \cite[Theorem 1.1]{CV}, there exists a global attractor 
$$
	\mathscr{A} = \gw (B^*) = \bigcap_{\tau \geq 0}\; \overline{\bigcup_{t \geq \tau} S(t) B^*}\, ,
$$
which is the $\gw$-limit set of an absorbing set for this mHNN semiflow.
\end{proof}

\section{\textbf{Approximate Synchronization of Memristive HNN}} 

In this section we shall prove the main result on approximate synchronization of the memristive Hopfield neural networks (mHNN) described by the model \eqref{Meq}. 

For the memristive Hopfield neural network $\mathcal{NW}$ with the model \eqref{Meq}, we define the interneuron gap function between any two neuron nodes $\mathcal{N}_i$ and $\mathcal{N}_j$ to be 
\begin{gather*}
	U_{ij} (t) = u_i(t) - u_j (t), \quad \text{for} \;\;  1\leq i, j \leq m. 
\end{gather*}
The differencing equations for $U_{ij}(t)$ are
\beq \bl{deq} 
	\begin{split}
	\frac{dU_{ij}}{dt} =&\, - \left(a_i U_{ij} + (a_i - a_j) u_j\right) + \sum_{\ell =1}^m \, \left(w_{i\ell} - w_{j\ell}\right) f_\ell (u_\ell)  \\
	&\,+ k\, (\vp_i (\rho) u_i - \vp_j (\rho) u_j) + (J_i - J_j) - P U_{ij} \sum_{\ell=1}^m \Gamma (u_\ell).
	\end{split}
\eeq
The related parameters involved in the differencing equations are denoted by
\begin{gather*}
	a^* = \max \{|a_i - a_j|: 1 \leq i, j \leq m \}, \quad W^* = \max \{ |w_{i\ell} - w_{j\ell}|: 1 \leq i, j, \,\ell \leq m\},   \\[2pt]
	\eta^* = \max \{|\eta_i - \eta_j|: 1 \leq i, j \leq m\},  \quad   J^* = \max \{|J_i - J_j|: 1 \leq i, j \leq m\}.
\end{gather*}
Now we prove the main result of this paper.
\begin{theorem} \bl{ThM}
	The memristive Hopfield neural network $\mathcal{NW}$ presented by the model \eqref{Meq} with Assumption \eqref{Asp} is approximate exponenially synchronizable. For any prescribed gap $\ve > 0$, there exists a threshold $P^*$ of network coupling strength,
\beq \bl{Pve}
	P^*(\ve) = \frac{1}{m \ve} \left( m W^* \gb + a^* Q^{1/2} + k \eta^* Q^{3/2} + J^* \right) \left[1 + \exp \{r (\sqrt{Q} + |V|)\} \right] ,
\eeq 
such that if the interneuron coupling strength $P > P^*(\ve)$, then the neural network $\mathcal{NW}$ is approximately synchronized, 
\beq \bl{aas}
	\deg_s \,(\mathcal{NW})= \sup_{g^0 \in Z} \left\{ \max_{1 \leq i < j \leq m} \left\{\limsup_{t \to \infty} \, |u_i (t) - u_j(t)| \right\}\right\} < \ve
\eeq
at a uniform exponential convergence rate
\beq \bl{rate}
	\mu (P) =    (a - k) + \frac{m P }{1 + \exp \left\{r (\sqrt{Q} + |V|)\right\}} > 0.
\eeq
\end{theorem}

\begin{proof}
In this proof we shall denote by $U(t) = U_{ij}(t)$ for a given pair of indices $i, j$. 

Step 1. For any given $1 \leq i \neq j \leq m$, multiply the equation \eqref{deq} with $U_{ij} (t) = U(t)$ and by Assumptions \eqref{Asp} to get
\beq \bl{Ut} 
	\begin{split}
	\frac{1}{2}&\, \frac{d}{dt}\, U^2 (t) = - a_i \,U^2(t) - (a_i - a_j) u_j (t) U(t) + \sum_{\ell =1}^m \,\left(w_{i\ell} - w_{j\ell}\right) f_\ell (u_\ell)\, U(t)   \\[2pt]
	&\,  + k (\vp_i (\rho) - \vp_j (\rho))u_i(t) U(t) + k \vp_j (\rho) (u_i(t) - u_j(t))U(t)  \\[2pt]
	&\,+ (J_i - J_j) U(t) - PU^2 (t) \, \sum_{\ell=1}^m \, \Gamma (u_\ell (t))  \\[2pt]
	\leq & - a U^2(t) + a^* |u_j(t) U(t)| + m W^* \gb |U(t)| + k [( 1 - \eta_i \rho^2) - (1 - \eta_j \rho^2)] u_i(t) U(t) \\[2pt]
	&\, + k (1 - \eta_j \rho^2(t)) (u_i(t) - u_j(t)) U(t) + J^* |U(t)| - PU^2 (t) \,\sum_{\ell=1}^m \,\Gamma (u_\ell (t))  \\[2pt]
	= &\, - (a - k)\,U^2(t) + (m W^* \gb + J^*)|U(t)| + a^* |u_j(t)| |U(t)| + k\, (\eta_j - \eta_i) \rho^2(t) u_i(t) U(t)  \\[2pt]
	&\,- k\,\eta_j \,\rho^2(t) U^2 (t) - PU^2 (t)\, \sum_{\ell=1}^m \, \Gamma (u_\ell(t))   \\[2pt]
	\leq &\,- (a - k)\,U^2(t) + (m W^* \gb + J^*)|U(t)| + a^* |u_j(t)| |U(t)|  + k\, \eta^* \rho^2(t) |u_i(t)| |U(t)|  \\[2pt]
	&\, - PU^2 (t) \,\sum_{\ell=1}^m \, \Gamma (u_\ell(t)), \quad \text{for} \;\; t > 0.
	\end{split}
\eeq 

Step 2. Now we treat the third term and the fourth term on the rightmost side of the differential inequality \eqref{Ut}. According to Theorem \ref{T2} and the finite-time bounding estimate \eqref{lsp}-\eqref{TB}, we have
$$
	 \limsup_{t \to \infty} \left[ \sum_{i=1}^m |u_i (t)|^2 + |\rho (t)|^2 \right] < Q, 
$$
where $Q > 0$ is a uniform constant given in \eqref{K}. Thus for any given initial state $g^0 \in \mathbb{R}^{m+1}$ there exists a finite time $T(g^0) \geq 0$, which may depend on $g^0$, such that the solution $g(t; g^0) = (u_1(t), \cdots, u_m(t), \rho(t))$ is ultimately bounded:
\beq \bl{bd}
	\sum_{i = 1}^m |u_i (t)|^2  + |\rho (t)|^2 < Q, \quad \text{for} \;\; t > T(g^0). 
\eeq 
Therefore, we have
\beq \bl{Qd}
	 a^* |u_j(t)| < a^* Q^{1/2} \;\;  \text{and} \; \; \rho^2(t) |u_i(t)| < Q^{3/2}, \;\; \text{for} \;\; t > T(g^0), \;\; 1 \leq i, j \leq m.
\eeq 
Substitute \eqref{Qd} in two terms  $a^* |u_j(t)| |U(t)|$ and $k\, \eta^* \rho^2(t) |u_i(t)| |U(t)|$ on the rightmost side of \eqref{Ut}. We see that, for any given initial state $g^0$, the solution $U(t) = U_{ij}(t)$ of the differencing equation \eqref{deq} satisfies the following differential inequality,
\beq \bl{Ub}
	\begin{split}
	&\frac{1}{2} \,\frac{d}{dt}\, U^2(t) + PU^2 (t)\, \sum_{\ell=1}^m \Gamma (u_\ell(t))   \\
	\leq - (a - k)\, U^2 (t) &\,+ (m W^* \gb + J^*)\, |U(t)| + (a^* Q^{1/2} + k \eta^* Q^{3/2}) |U(t)|   \\[5pt]
	= - (a - k)\, U^2 (t) &\,+ (m W^* \gb + a^* Q^{1/2} + k \eta^* Q^{3/2} + J^*) |U(t)|,  \;  \text{for} \; t > T(g^0).
	\end{split}
\eeq

Next we handle the term $PU^2 (t) \sum_{\ell=1}^m \Gamma (u_\ell(t))$ in \eqref{Ub}, which can be called the weak interneuron coupling \cite{YT}. Since \eqref{bd} implies that  $ |u_\ell (t)| \leq \sqrt{Q}$ for $t > T(g^0)$ and for all $1 \leq \ell \leq m$, we have
\beq \bl{Gb}
	\sum_{\ell = 1}^m \Gamma (u_\ell (t)) = \sum_{\ell = 1}^m \frac{1}{1 + \exp [- r (u_\ell (t) - V)]} \geq \frac{m}{1 + \exp \left\{r (\sqrt{Q} + |V|)\right\}}
\eeq 
for $t > T(g^0)$. Substitute \eqref{Gb} in the inequality \eqref{Ub}. Then we get
\beq \bl{UP}
	\begin{split}
	&\,\frac{d}{dt} U^2(t) + 2 \left[ (a - k) + \frac{m P }{1 + \exp \left\{r (\sqrt{Q} + |V|)\right\}} \right] U^2 (t)  \\[4pt]
	\leq &\, 2 \left(m W^* \gb  + a^* Q^{1/2} + k \eta^* Q^{3/2} + J^* \right) |U(t)|, \quad  \text{for} \;\; t > T(g^0).
	\end{split}
\eeq

Step 3. Finally we have to tackle the nonlinear Gronwall inequality \eqref{UP}. Denote by 
\beq \bl{mu}
	 \mu =  (a - k) + \frac{m P }{1 + \exp \left\{r (\sqrt{Q} + |V|)\right\}},
\eeq
Then \eqref{UP} leads to the following differential inequality
\begin{equation*}
	\begin{split}
	&\frac{d}{dt} \,U^2(t) + 2\mu \,U^2(t) \leq 2 \left(m W^* \gb + a^* Q^{1/2} + k \eta^* Q^{3/2} + J^* \right) |U(t)| \\
	\leq &\, \mu \,U^2(t) + \frac{1}{\mu} \left(m W^* \gb + a^* Q^{1/2} + k \eta^* Q^{3/2} + J^* \right)^2,  \quad t > T(g^0).
	\end{split}
\end{equation*}
Hence it follows that
\beq \bl{gw}
	\begin{split}
	&\frac{d}{dt} \,U^2(t) + \mu \,U^2(t) \leq \frac{\left(m W^* \gb + a^* Q^{1/2} + k \eta^* Q^{3/2} + J^* \right)^2}{(a - k) + \frac{m P }{1 + \exp \left\{r (\sqrt{Q} + |V|)\right\}}} \\
        = &\,\frac{\left(m W^* \gb + a^* Q^{1/2} + k \eta^* Q^{3/2} + J^* \right)^2 \left(1 + \exp \left\{r (\sqrt{Q} + |V|)\right\}\right)}{(a - k)\left(1 + \exp \left\{r (\sqrt{Q} + |V|)\right\}\right) + m P}, \;\,  t > T(g^0).
	\end{split}
\eeq
Thus for any prescribed gap $\ve > 0$, if the threshold condition $P > P^*(\ve)$ is satisfied, then the gap function $U(t) = U_{ij}(t) = u_i(t) - u_j(t) $ for any two neurons $\mathcal{N}_i$ and $\mathcal{N}_j$ satisfies 
\beq \bl{Gw}
	\begin{split} 
	U^2(t) &\,\leq e^{- \mu (t - T(g^0))} U^2 (T(g^0))   \\[3pt]
	&\, + \frac{1}{\mu} \left[\frac{(m W^* \gb + a^* Q^{1/2} + k \eta^* Q^{3/2} + J^*)^2 (1 + \exp \{r (\sqrt{Q} + |V|)\})}{(a - k)(1 + \exp \left\{r (\sqrt{Q} + |V|)\right\})+ mP}\right]  \\
	&\, = e^{- \mu (t - T(g^0))} U^2 (T(g^0))   \\[3pt]
	&\, + \left[\frac{(m W^* \gb + a^* Q^{1/2} + k \eta^* Q^{3/2} + J^*) (1 + \exp \{r (\sqrt{Q} + |V|)\})}{(a - k)(1 + \exp \left\{r (\sqrt{Q} + |V|)\right\})+ mP}\right]^2 ,
	\end{split}
\eeq
for $t > T(g^0), 1 \leq i \neq j \leq m$, and for any initial state $g^0$. 

For any prescribed gap $\ve > 0$, if the threshold condition $P > P^*(\ve)$ shown by \eqref{Pve} in this theorem is satisfied, then the inequality \eqref{Gw} shows that
\beq \bl{ve}
	\begin{split}
	& \limsup_{t \to \infty} |u_i (t) - u_j(t)| = \limsup_{t \to \infty} |U(t)|   \\[3pt]
	\leq &\, \frac{ (m W^* \gb + a^* Q^{1/2} + k \eta^* Q^{3/2} + J^*)(1 + \exp \{r (\sqrt{Q} + |V|)\})}{(a - k)(1 + \exp \left\{r (\sqrt{Q} + |V|)\right\})+ mP}   \\
	< &\,\frac{1}{mP^*} (m W^* \gb + a^* Q^{1/2} + k \eta^* Q^{3/2} + J^*)(1 + \exp \{r (\sqrt{Q} + |V|)\}) = \ve
	\end{split}
\eeq
for any $1 \leq i \neq j \leq m$ and any initial state $g^0$. \eqref{ve} means that we have shown 
$$
	\deg_s (\mathcal{NW}) = \sup_{g^0 \in Z} \left\{ \max_{1 \leq i < j \leq m} \left\{\limsup_{t \to \infty} \, |u_i (t) - u_j(t)| \right\}\right\} < \ve
$$
if the interneuron coupling strength exceeds the threshold, $P > P^*(\ve)$. Thus the statement \eqref{aas} and the exponential convergence rate $\mu$ provided in \eqref{rate} are proved. According to the definition \eqref{gap}, the proof is completed.
\end{proof}

\begin{corollary}
	If the weak synaptic coupling $ - Pu_i \, \sum_{j=1}^m \Gamma (u_j)$ in the model \eqref{Meq} is replaced by the linear synaptic coupling $- P \sum_{j=1}^m (u_i - u_j)$ or called strong network coupling, then the approximate synchronization stated in Theorem \ref{ThM} remains valid with the simpler threshold 
\beq \bl{wP}
	P^*(\ve) =  \frac{1}{m\, \ve} \left(m W^* \gb + a^* Q^{1/2} + k \eta^* Q^{3/2} + J^* \right)
\eeq
and the changed convergence rate $\mu = a - k + P$.
\end{corollary}
\begin{proof}
	The linear network coupling $- P \sum_{j=1}^m (u_i - u_j)$ will only cause the following change of the last term on the right-hand side of the first inequality in \eqref{ui},
\beq \bl{lcp}
	\begin{split}
	&\, C_1 \sum_{i=1}^m \sum_{j=1}^m P(u_i - u_j)u_i = \sum_{1 \leq i < j \leq m} C_1 P \left[(u_i - u_j) u_i + (u_j - u_i) u_j\right]   \\
	= &\, \sum_{1 \leq i < j \leq m} C_1 P(u_i - u_j)^2 = \frac{1}{2}\, C_1 P \sum_{i=1}^m \sum_{j=1}^m (u_i - u_j)^2.
	\end{split}
\eeq
The rest of proof can be verified through the steps by simple recalculation.
\end{proof}

\section{\textbf{Extension to Memristive Hopfield- Hebbian Neural Networks}}

In this section we shall extend the study of approximate synchronization to memristive Hopfield neural networks with Hebbian learning rules. 

There have been growing interests in recent years to adopt biologically inspired learning rules and algorithms to train artificial neural networks, especially for improving the efficiency and accuracy related to the backpropagation process. Hebbian learning rules \cite{H, VB} update the synaptic weights in neural networks by dynamical equations and have demonstrated good performance in unsupervised machine learning tasks on certain datasets like MNIST and CIFAR for training in image recognition and classification \cite{PC, VFR}. 

Here we consider a memristive Hopfield-Hebbian neural network composed of $m$ neuron nodes, denoted by $\mathcal{NWH}$, and governed by the following model:
\beq \bl{Heq}
	\begin{split}
	\frac{d u_i}{d t} & = - a_i u_i + \sum_{j =1}^m w_{ij}  f_j (u_j) + k_i\, \psi_i (\rho) u_i + J_i - P \sum_{j=1}^m (u_i - u_j), \;\; 1 \leq i \leq m,  \\
	\frac{d w_{ij}}{d t} & = - c_{ij}\, w_{ij} + \gl_{ij} \, f_i (u_i) f_j (u_j), \quad  1 \leq i, j \leq m,  \\[4pt]
	\frac{d \rho}{d t} & = \sum_{i=1}^m \ga_i u_i - b \rho,  \quad  t > 0.
	\end{split} 
\eeq
The $u_i$-equations and the $\rho$ equation here in \eqref{Heq} are similar to \eqref{Meq} but the synaptic coupling weights $w_{ij}(t)$ in this model is dynamically updated and governed by the $w_{ij}$ equation, which is called Hebbian learning rule. The case with all coefficients $\gl_{ij} > 0$ is called Hebbian learning and the case with all $\gl_{ij} < 0$ called anti-Hebbian learning. The initial conditions for the system \eqref{Heq} are denoted by $u_i(0) = u_i^0, \, w_{ij}(0) = w_{ij}^0, \, \rho (0) = \rho^0$. It can be denoted as a vector $g^0 = (u_i^0, w_{ij}^0, \rho^0)$. The state space of this memristive Hopfield-Hebbian neural network turns out to be $\mathbb{R}^{m(1+m) +1}$.

Several remarks can be made in regard to the versatility and features about this memristive Hopfield-Hebbian neural network model \eqref{Heq}, ncluding the model \eqref{Meq}.

Remark 1. The activation functions $\{f_i(x): 1 \leq i \leq m\}$ in this model \eqref{Heq} can well be heterogeneous for different neurons in one neural network. Typical activation functions \cite{OC} are continuous differentiable, monotone, and bounded functions except a few like Softplus $Sp (x) = \ln \, (1 + e^x)  < x, ReLU (x)$ and $GeLU (x) = (x/2)[1 + Erf (x/\sqrt{2})]$, which are nevertheless quasilinear. Since the transmembrane potential of biological neurons are uniformly bounded from the resting potentials typically $-60$ \emph{mv} up to the bursting limit about $+40$ \emph{mv}, one can still plausibly assume these quasilinear activation functions to be uniformly bounded.

Remark 2. There are several typical memristor window functions: sigmoidal function $\tanh (s)$, linear function $1 - \eta | s |$, quadratic function $1 - \eta s^2$, Jogelker memristor $1 - (2s -1)^{2q}$, and Strukov-Williams memristor $s (\eta - s)$, where $\eta > 0$ is a normalized parameter and $q$ is a positive integer. Here in this model \eqref{Heq} $\psi_i(\rho)$ can take only one these types \cite{VM} but maybe different parameters for different neuron nodes.  

Remark 3. Concerning the synaptic weights dynamics governed by Hebbian rules in the $w_{ij}$-equations of \eqref{Heq}, we emphasize that the Hebbian equations incarnate the key factors of locality, synaptic dissipativity, cooperative or competitive tendency possessed in biological neuronal networks. The coefficient matrix $(\gl_{ij})_{m x m}$ in this model describes the topological connection in the network and likely to be a sparse matrix reflecting the multi-layer structure of deep learning.

We make the following assumptions on this model \eqref{Heq}: 
\beq \bl{Ast}
	\begin{split}
	&a =  = \min \{a_i: 1 \leq i \leq m\}  > \frac{1}{2}\, k \,\eta^2, \quad  c_{ij} > 0,\; \; \gl_{ij} \in \mathbb{R}, \; 1 \leq i, j \leq m, \\[2pt]
        \eta = &\, \min \{\eta_i: 1 \leq i \leq m\} > 0, \quad  |f_i(s)| \leq \gb, \quad \psi_i (s) = s(\eta_i - s), \; 1 \leq i \leq m.
	\end{split}
\eeq
The descriptions on the rest parameters are same as in Section 1. We shall prove the extension result of approximate synchronization for the memristive Hopfield-Hebbian neural network $\mathcal{NWH}$. The involved parameters in the next result are listed below. 
\begin{gather*}
	a^* = \max \{|a_i - a_j|: 1 \leq i, j \leq m \}, \quad W_0^* = \max \{ |w_{i\ell} (0) - w_{j\ell} (0)|: 1 \leq i, j, \,\ell \leq m\},   \\[2pt]
	\eta^* = \max \{|\eta_i - \eta_j|: 1 \leq i, j \leq m\},  \quad   J^* = \max \{|J_i - J_j|: 1 \leq i, j \leq m\}.
\end{gather*}

\begin{theorem} \bl{TMH}
	The memristive Hopfield-Hebbian neural network $\mathcal{NWH}$ presented by the model \eqref{Heq} with Assumption \eqref{Ast} is approximate exponenially synchronizable. For any prescribed gap $\ve > 0$, there exists a threshold $P^*$ of network coupling strength,
\beq \bl{Ph}
	P^*(\ve) = \frac{1}{m\, \ve} \left[ a^* G^{1/2} + k \eta^* G^{3/2} + 2 m \gb \sqrt{1 + \frac{\gl^2 \gb^4}{c^2}} + J^* \right],
\eeq 
where $G$ is shown in \eqref{lsp}, such that if the interneuron coupling strength $P > P^*(\ve)$, then the neural network $\mathcal{NWH}$ is approximately synchronized, 
\beq \bl{asd}
	\deg_s \,(\mathcal{NWH}) = \sup_{g^0 \in Z} \left\{ \max_{1 \leq i < j \leq m} \left\{\limsup_{t \to \infty} \, |u_i (t) - u_j(t)| \right\}\right\} < \ve
\eeq
at a uniform exponential convergence rate

\beq \bl{Rte}
	\xi (P) = a - \frac{1}{2}\, k\, \eta + mP.
\eeq
\end{theorem}

\begin{proof}
It will go through two steps. The first step aims to outcropping a uniform ultimate bound for all the solutions of this Hopfield-Hebbian neural network model. The second step will carry out the analysis of the interneuron differencing equations for the neuron gap functions to reach the result of approximate exponential synchronization.

	Step 1. First we want to estimate an ultimate bound for all solutions of this neural network model. Multiply the $w_{ij}$-equation in \eqref{Heq} by $w_{ij}(t)$ and by \eqref{Ast} to get
\beq \bl{wd}
	\begin{split}
	\frac{1}{2} \,\frac{d}{dt}\, w^2_{ij}(t) & = - c_{ij}\, w^2_{ij}(t) + \gl_{ij} \,w_{ij}(t) f_i(u_i(t)) f_j(u_j(t))   \\
	& \leq - c \,w^2_{ij}(t) + \gl \,\gb^2\, w_{ij}(t) \leq - \frac{1}{2}\, c\, w^2_{ij}(t) + \frac{1}{2 c}\, \gl^2 \gb^4,
	\end{split}
\eeq
where $\gl = \max \{|\gl_{ij}| \}$ and $c = \min \{c_{ij}: 1 \leq i, j \leq m\} > 0$.  It follows that
\beq \bl{wt}
	w^2_{ij}(t) \leq e^{- c\, t} \,w^2_{ij}(0) + \frac{\gl^2 \gb^4}{c^2} \leq w^2_{ij}(0) + \frac{\gl^2 \gb^4}{c^2} \quad t > 0.
\eeq
Sum up the $u_i$-equations multiplied by $C_3 u_i(t)$ and the $\rho$-equation multiplied by $\rho(t)$ in \eqref{ur}. Without loss of generality, one can take all the initial weights $w_{ij}(0) \in \{0, 1\}$ depending on the connectivity between the two neuron nodes. Note that $\psi_i(\rho) = \rho (\eta_i  - \rho) \leq \frac{1}{2} \eta^2 - \frac{1}{2} \rho^2$ and by interneuron coupling treatment \eqref{lcp}, we have
\beq \bl{ur}
	\begin{split}
	& \frac{1}{2} \frac{d}{dt} \left( \sum_{i = 1}^m C_3 u^2_i + \rho^2 \right) \leq \frac{1}{2}\, \frac{d}{dt}\, \left(\sum_{i = 1}^m C_3 u^2_i + \rho^2 \right) + \frac{1}{2}\, C_3 P \,\sum_{i=1}^m \sum_{j=1}^m \, (u_i - u_j)^2  \\
	= &\,\sum_{i=1}^m C_3 \left[- a_i u^2_i  + \sum_{j =1}^m w_{ij}(t) f_j (u_j) u_i + k \psi_i (\rho) u_i^2 +  J_i u_i \right] + \sum_{i=1}^m \ga_i \,u_i \rho - b \rho^2   \\
	\leq &\,\sum_{i=1}^m C_3 \left[- a u_i^2(t) + m \left(w^2_{ij}(0) + \frac{\gl^2 \gb^4}{c^2} \right)^{1/2} \gb |u_i (t)| \right]   \\
	&\, + \sum_{i=1}^m C_3 \left[ \frac{1}{2} k \left(\eta^2 - \rho^2(t)\right) u_i^2(t) + J |u_i(t)| \right] + \sum_{i=1}^m \frac{\ga_i^2}{2b} u_i^2 (t) - \frac{b}{2} \rho^2(t)  \\
	\leq &\, - \sum_{i=1}^m C_3\left(a - \frac{1}{2} k \eta^2\right) u_i^2(t) + \sum_{i=1}^m \frac{\ga_i^2}{2b} u_i^2 (t) - \frac{b}{2} \rho^2(t)  \\
	&\, + \sum_{i=1}^m C_3 \left( J + m \gb \sqrt{1 + \frac{\gl^2 \gb^4}{c^2}}\, \right) |u_i (t)| \\
	\leq &\, - \frac{1}{2} \sum_{i=1}^m C_3\left(a - \frac{1}{2} k \eta^2\right) u_i^2(t) + \frac{1}{2} \left[ \sum_{i=1}^m \frac{\ga_i^2}{b} u_i^2 (t) - b \rho^2(t)\right]  \\
	&\, + \frac{1}{2} \,C_3 \left( J + m \gb \sqrt{1 + \frac{\gl^2 \gb^4}{c^2}} \,\right)^2 \left(a - \frac{1}{2} k \eta^2\right)^{-1} ,
	\end{split}
\eeq
where $C_3 > 0$ is a scaling constant. By the Assumption \eqref{Ast}, we choose the positive scaling constant
\beq \bl{C3}
	C_3 = \frac{\frac{m \ga^2_i}{b} + \frac{1}{2} b}{a - \frac{1}{2} k \eta^2}  \; \quad \text{so that} \quad C_3\left(a - \frac{1}{2} k \eta^2\right) - \frac{m \ga^2_i}{b} = \frac{b}{2}\, .
\eeq
Substitute \eqref{C3} in \eqref{ur} to obtain
\beq \bl{UR}
	\frac{d}{dt} \left(\sum_{i = 1}^m C_3 u_i^2(t) + \rho^2(t)\right) + b \left(\sum_{i=1}^m u_i^2(t) + \rho^2(t) \right) \leq C_4, \quad  \text{for} \;\; t > 0,
\eeq
where 
\beq \bl{C4}
        C_4 = \left(\frac{m \ga^2_i}{b} + \frac{b}{2} \right) \left( J + m \gb \sqrt{1 + \frac{\gl^2 \gb^4}{c^2}} \,\right)^2 \left(a - \frac{1}{2} k \eta^2\right)^{- 2}	
\eeq
is a uniform constant independent of any initial state. The differential inequality \eqref{UR} can be written as
\beq \bl{ub}
	\frac{d}{dt} \left(\sum_{i = 1}^m C_3 u_i^2(t) + \rho^2(t)\right) + \gs \left(\sum_{i=1}^m C_3 u_i^2(t) + \rho^2(t) \right) \leq C_4, \quad \text{for} \;\; t > 0,
\eeq
where 
\beq \bl{dr}
	\gs = b \, \min \left\{\frac{1}{C_3},\, 1\right\}.
\eeq
We can solve this differential Gronwall inequality to confirm 
$$
	\min \{C_3, 1\} \left[\sum_{i = 1}^m u_i^2(t) + \rho^2(t)\right] \leq \sum_{i = 1}^m C_3 u_i^2(t) + \rho^2(t) \leq e^{- \gs t} \left[\sum_{i=1}^m C_3 u_i^2(0) + \rho^2(0) \right] + \frac{C_4}{\gs}
$$ 
which shows that 
\beq  \bl{lsp}
	\limsup_{t \to \infty} \left( \sum_{i = 1}^m u_i^2(t) + \rho^2(t) \right) < G = 1 + \frac{C_4}{\gs\, \min \{C_3, 1\}}. 
\eeq 
Here the $C_3, C_4$ and $G$ are uniform positive constants independent of any initial state. 

Step 2. Next we shall carry out the analysis of the dynamic gap functions $V_{ij} (t) = u_i(t) - u_j(t)$ among the neurons, which satisfy the following differencing equations,
\beq \bl{Veq}
	\begin{split}
	\frac{dV_{ij}}{dt} = &\, - (a_i V_{ij} + (a_i - a_j) u_j ) + \sum_{\ell =1}^m \, \left(w_{i\ell} - w_{j\ell}\right) f_\ell (u_\ell)   \\
	&\, + k\, (\psi_i (\rho) u_i - \psi_j (\rho) u_j)  + J_i - J_j - mPV_{ij}, \;\;  t > 0,
	\end{split}
\eeq
where the last term comes from  
$$
	 P \left( \sum_{\ell =1}^m (u_i - u_\ell) - \sum_{\ell=1}^m (u_j - u_\ell)\right) = P\, \sum_{\ell = 1}^m (u_i - u_j) = mPV_{ij}.
$$

For any given $1 \leq i \neq j \leq m$, multiply the $V_{ij}$-equation by $V_{ij}(t)$. In the sequel we shall denote $V_{ij}(t)$ by $V(t)$ for notational simplicity. By Assumption \eqref{Ast} and \eqref{wt} with the setting $w_{ij}(0) \in \{0, 1\}$, we obtain

\begin{equation*}
	\begin{split}
	&\frac{1}{2}\, \frac{d}{dt} V^2 (t) = - a_i \,V^2 (t) - (a_i - a_j) u_j (t) V(t) + \sum_{\ell =1}^m \,\left(w_{i\ell}(t) - w_{j\ell}(t) \right) f_\ell (u_\ell) V(t)   \\[5pt]
	& + k (\psi_i (\rho) - \psi_j (\rho))u_i(t) V(t) + k \psi_j (\rho) V^2(t) + (J_i - J_j) V(t) - mPV^2 (t)  \\[5pt]
	\leq & - a V^2(t) + a^* |u_j(t)| |V(t)| + 2 m \gb \sqrt{1 + \frac{\gl^2 \gb^4}{c^2}} \,|V(t)|   \\[10pt]
	&\, + k (\eta_i - \eta_j) \rho(t) |u_i(t)| |V(t)| + k \left(\rho(t) \eta_j - \rho^2 (t) \right) V^2(t).+ J^* |V(t)| - m PV^2 (t)  \\[6pt]
	\leq & - a V^2(t) + a^* |u_j(t)| |V(t)| + 2 m \gb \sqrt{1 + \frac{\gl^2 \gb^4}{c^2}} \,|V(t)|   \\[6pt]
	&\, + k \eta^* |\rho(t)| |u_i(t)| |V(t)| + \frac{1}{2}\, k \,(\eta - \rho^2 (t)) V^2(t).+ J^* |V(t)| - m PV^2 (t)  \\[2pt]
	\leq &\, - \left(a - \frac{1}{2} k \eta + mP \right)\,V^2(t)   \\
	&\, + \left(a^* |u_j(t)| + k \eta^* |\rho(t)| |u_i(t)| + 2 m \gb \sqrt{1 + \frac{\gl^2 \gb^4}{c^2}} + J^* \right) |V(t)|, \;\;  \text{for} \;\; t > 0.
	\end{split}
\end{equation*}
According to the asymptotic estimate \eqref{lsp}, for any given initial state $g^0 = (u_i^0, w_{ij}^0, \rho^0)$ of the neural network, there is a finite time $T(g^0) > 0$ such that 
$$
	 \sum_{i = 1}^m u_i^2(t) + \rho^2(t) < G, \quad \text{for} \;\; t > T_{g^0}.
$$
Therefore the above shown differential inequality implies that for $t > T_{g^0}$,
\begin{equation*}
	\begin{split}
	& \frac{d}{dt} V^2 (t) + 2 \left(a - \frac{1}{2}\, k \eta + mP \right)V^2(t)    \\[8pt]
	\leq &\, 2 \left(a^* G^{1/2} + k \eta^* G^{3/2} + 2 m \gb \sqrt{1 + \frac{\gl^2 \gb^4}{c^2}} + J^* \right) |V(t)|   \\
	\leq &\,\left(a - \frac{1}{2} k \eta + mP \right)V^2(t) + \frac{\left(a^* G^{1/2} + k \eta^* G^{3/2} + 2 m \gb \sqrt{1 + \frac{\gl^2 \gb^4}{c^2}} + J^* \right)^2}{a - \frac{1}{2} k \eta + mP}.
	\end{split}
\end{equation*}
It follows that

\beq  \bl{VG}
	\begin{split}
	& \frac{d}{dt} V^2 (t) + \left(a - \frac{1}{2}k \eta + mP \right)V^2(t)    \\
	\leq &\, \left(a - \frac{1}{2}k \eta + mP \right)^{-1} \left[a^* G^{1/2} + k \eta^* G^{3/2} + 2 m \gb \sqrt{1 + \frac{\gl^2 \gb^4}{c^2}} + J^* \right]^2, 
	\end{split}
\eeq
for $t > T_{g^0}$. Denote by 
\beq \bl{aph}
	\xi = a - \frac{1}{2} k \eta + mP.
\eeq
Integration of the Gronwall inequality \eqref{VG} on the time interval $[T_{g^0}, t]$ shows that for any initial state $g^0$ and any pair of indices $1 \leq i \neq j \leq m$,
\beq \bl{Vt}
	V^2(t) \leq e^{\xi(t - T_{g^0})} V^2(T_{g^0}) + \left[\frac{a^* G^{1/2} + k \eta^* G^{3/2} + 2 m \gb \sqrt{1 + \frac{\gl^2 \gb^4}{c^2}} + J^*}{a - \frac{1}{2}k \eta + mP}\right]^2, \;\;  t > T_{g^0}.
\eeq 
For any prescribed small gap $\ve > 0$, if the threshold condition $P > P^*(\ve)$ shown in \eqref{Ph} is satisfied, then from \eqref{Vt} we can conclude that for any $1 \leq i \neq j \leq m$,
\beq \bl{Vsp}
	\begin{split}
	&\limsup_{t \to \infty} |u_i(t) - u_j(t)| = \limsup_{t \to \infty} |V_{ij}(t)| = \limsup_{t \to \infty} |V(t)|  \\[4pt]
	 \leq &\,\frac{a^* G^{1/2} + k \eta^* G^{3/2} + 2 m \gb \sqrt{1 + \frac{\gl^2 \gb^4}{c^2}} + J^*}{a - \frac{1}{2}k \eta + mP}  \\
	 < &\, \frac{1}{m P^*} \left( a^* G^{1/2} + k \eta^* G^{3/2} + 2 m \gb \sqrt{1 + \frac{\gl^2 \gb^4}{c^2}} + J^* \right) = \ve .
	\end{split}
\eeq
Consequently \eqref{asd} is proved. The proof is completed.
\end{proof}

\textbf{Conclusions}

We summarize the contributions in this paper and outlook for further researches.

In view of the fact that complete synchronization to zero gap for artificial neural networks is theoretically unrealistic due to multiple factors of mismatched or time-varying synaptic  weights and heterogenous activation or coupling components, we proposed a quantitative concept of approximate synchronization for neural networks in Definition 1.1 of this paper. 

We introduced and investigated two new model. One is Hopfield neural networks with nonlinear memristor and weak interneuron coupling and another one is memristive Hopfield neural networks with weights dynamically updated by Hebbian learning rules, which have good performance in unsupervised data training. Taking the approach of scaled\emph{ a priori} estimates of solutions, for both models we proved the existence of an absorbing set with a sharp ultimate bound, which paves the way to exploration of asymptotic dynamics.

In the main result Theorem \ref{ThM} and the extension result Theorem \ref{TMH} of this paper we established a sufficient threshold condition to be satisfied by a single interneuron coupling strength coefficient for achieving approximate exponential synchronization of the neural network respectively for each model. The merit and the usefulness of these results lie in the explicit expression of the threshold and the exponential convergence rate, which are computable in terms of the original parameters in advance. Moreover, the asymptotic gap can be prescribed arbitrarily. 

The rigorously proved theoretical results along with the methodology in this work can be applied and generalized to many other types of artificial neural networks in unsupervised machine learning with mismatched parameters, heterogeneous activation devises, non-identical memristors, or complex synaptic couplings. 

Mathematical research on approximate synchronization and predicable asymptotic dynamics of artificial neural networks is meaningful in deep learning performance  but remains widely open, especially for  tackling the backpropagation operations with all kinds of models, datasets, and algorithms.

\end{document}